\documentclass[envcountsect,referee]{svjour3}

\smartqed  % flush right qed marks, e.g. at end of proof
\usepackage{graphicx}
\usepackage{latexsym,bm,amsmath,amssymb,mathrsfs}
\usepackage{rotating}
\usepackage{multirow,booktabs,cases}
\usepackage[misc]{ifsym}
\usepackage[style=1]{mdframed}
\usepackage{algorithm,algorithmic}
\usepackage[colorlinks=true]{hyperref}
\hypersetup{urlcolor=blue,citecolor=blue,linkcolor=blue}
\usepackage{epstopdf}

\def\[{\begin{equation}}
\def\]{\end{equation}}

\newtheorem{assum}{Assumption}[section]
\newtheorem{exam}{Example}[section]

\numberwithin{equation}{section}
\allowdisplaybreaks

\begin{document}
\graphicspath{{./PIC/}}

\title{Generalized tensor equations with leading structured tensors}

%\titlerunning{Relaxed projection method}     % if too long for running head

\author{Weijie Yan\and Chen Ling\and Liyun Ling \and Hongjin He}

%\authorrunning{F. Author \and S. Author} % if too long for running head

\institute{W. Yan\and C. Ling \and H. He \at
Department of Mathematics, School of Science, Hangzhou Dianzi University, Hangzhou, 310018, China.\\
\email{yanweijie1314@163.com}
\and  C. Ling  \at
\email{macling@hdu.edu.cn}
\and H. He (\Letter) \at
\email{hehjmath@hdu.edu.cn}
\and L. Ling \at
School of Information Engineering, Hangzhou Dianzi University, Hangzhou, 310018, China.
\email{lingliyun@163.com}
 }

\date{Received: date / Accepted: date}

\maketitle

\begin{abstract}
The system of tensor equations (TEs) has received much considerable attention in the recent literature. In this paper, we consider a class of generalized tensor equations (GTEs). An important difference between GTEs and TEs is that GTEs can be regarded as a system of non-homogenous polynomial equations, whereas TEs is a homogenous one. Such a difference usually makes the theoretical and algorithmic results tailored for TEs not necessarily applicable to GTEs. To study properties of the solution set of GTEs, we first introduce a new class of so-named ${\rm Z}^+$-tensor, which includes the set of all P-tensors as its proper subset. With the help of degree theory, we prove that the system of GTEs with a leading coefficient ${\rm Z}^+$-tensor has at least one solution for any right-hand side vector. Moreover, we study the local error bounds under some appropriate conditions. Finally, we employ a Levenberg-Marquardt algorithm to find a solution to GTEs and report some preliminary numerical results.
\end{abstract}


\section{Introduction}\label{Introd}
Let ${\cal A}=(a_{i_1i_2\cdots i_m})$ with $a_{i_1i_2\cdots i_m}\in \mathbb{R}$ for $i_1,i_2,\cdots,i_m \in [n]:=\{1,2,\cdots,n\}$ be an $m$-th order $n$-dimensional square real tensor and $b\in \mathbb{R}^n$. The system of tensor equations (or multi-linear system) investigated in the literature refers to the task of finding a vector $x\in{\mathbb R}^n$ such that
\begin{equation}\label{mulineq}
\mathcal{A}x^{m-1}=b,
\end{equation}
where $\mathcal{A}{x}^{m-1}$ is defined as a vector, whose $i$-th component
is given by
\begin{equation}\label{Axm-1}
(\mathcal{A}{x}^{m-1})_i=\sum_{i_2,\cdots,i_m=1}^na_{ii_2\cdots i_m}x_{i_2}\cdots x_{i_m}, \quad i=1,2,\cdots,n.
\end{equation}
Recently, it has been well-documented that the system of tensor equations \eqref{mulineq} arises in a number of applications such as data mining \cite{LN15}, numerical partial  differential equations \cite{DW16}, and tensor complementarity problems \cite{SQ16,XLX17}. Therefore,  the system of tensor equations \eqref{mulineq} has received much considerable attention in the recent literature, e.g., see \cite{DW16,Han17,HLQZ18,LXX17,LDG19,LLV18,LM18,XJW18} and references therein. Especially, in \cite{DW16},  Ding and Wei proved that, if the coefficient tensor $\mathcal{A}$ in (\ref{mulineq}) is a nonsingular M-tensor \cite{DQW13,Zhang3}, then the problem (\ref{mulineq}) has a unique positive solution for any given positive vector $b$ (i.e., each  component of $b$ is positive)  in $\mathbb{R}^n$, in addition to generalizing the Jacobi and Gauss-Seidel methods to find the unique solution. Since solving tensor equations system plays an instrumental role in engineering and scientific computing, many numerical methods have been developed to solve \eqref{mulineq} with M-tensors, e.g., see \cite{Han17,HLQZ18,LXX17,LLV18,XJW18}. However, the coefficient tensor ${\mathcal A}$ of \eqref{mulineq} arising from many real-world problems, such as data mining \cite{LN15}, tensor complementarity problems \cite{SQ16,XLX17} and high dimensional interpolations in the reproducing kernel Banach spaces \cite{Y18}, is often not a nonsingular M-tensor. Moreover, we observe that \eqref{mulineq} is a system of homogenous polynomial equations, but some applications usually lack of the underlying homogeneousness emerging in \eqref{mulineq}, for example the high-order Markov chains \cite{LN14} and multilinear PageRank problems \cite{GLY15}. Unfortunately, for the aforementioned two cases, it is unclear that whether \eqref{mulineq} has solutions for any vector $b\in \mathbb{R}^n$ when the coefficient tensor $\mathcal{A}$ is not a nonsingular M-tensor, and the numerical algorithms tailored for \eqref{mulineq} still work or not.

In this paper, we consider a class of so-named generalized tensor equations (GTEs), which can be written as
\begin{equation}\label{defMS}
{\cal A}_1x^{m-1}+{\cal A}_2x^{m-2}+\cdots+{\cal A}_{m-2}x^2+{\cal A}_{m-1}x = b,
\end{equation}
where ${\cal A}_k \in \mathbb{T}_{m-k+1,n}~(k=1,2,\cdots,m-1)$ and $b\in \mathbb{R}^n$. Here, we denote the set of all $l$-th order $n$-dimensional square real tensors by $\mathbb{T}_{l,n}$ for $l=2,3,\cdots,m$.
Apparently, \eqref{mulineq} falls into a special case of \eqref{defMS} with settings of ${\mathcal A_1}={\mathcal A}$ and ${\mathcal A}_i$ ($i=2,\cdots,m-1$) being zero tensors.

Although the system of GTEs is an interesting generalization of \eqref{mulineq}, to the best of our knowledge, there is no paper contributed to the solution existence of \eqref{defMS} with any right-hand vector $b\in{\mathbb R}^n$. Thus, the first contribution of this paper is to show that \eqref{defMS} has at least one solution for any $b\in{\mathbb R}^n$ when the leading coefficient tensor ${\mathcal A}_1$ is a ${\rm Z}^+$-tensor, where ${\rm Z}^+$-tensor is a newly introduced structured tensor in the paper, which includes the set of all P-tensors as its proper set. As a byproduct of our results, \eqref{mulineq} has a solution for any $b\in{\mathbb R}^n$ when ${\mathcal A}$ is a ${\rm Z}^+$-tensor, in addition to showing that \eqref{mulineq} has a unique solution if ${\mathcal A}$ is a strong P-tensor. Such a theoretical result is an interesting complement to \cite{DW16}. Moreover, we study the local error bounds under some appropriate conditions, which is the second contribution of this paper and plays an important role in algorithmic design. Since most of the recent numerical algorithms are devoted to \eqref{mulineq} with M-tensors, we employ an efficient Levenberg-Marquardt algorithm to find numerical solutions of the generalized tensors equations \eqref{defMS}. The computational results demonstrate that the proposed Levenberg-Marquardt algorithm is competitive to the state-of-art algorithms in \cite{Han17,HLQZ18,LXX17} when dealing with \eqref{mulineq} with M-tensors. More promisingly, the Levenberg-Marquardt algorithm performs well for both \eqref{mulineq} and \eqref{defMS} with generic tensors with a relatively high probability.

The remainder of the paper is organized as follows. In Section \ref{prelim}, we will summarize some definitions on tensors and introduce a new class of structured tensors, which includes many class of special tensors as its proper subset. In Section \ref{ExistRes}, by utilizing the topological degree theory, we will present an existence result on solutions for the system of generalized tensor equations (\ref{defMS}).  In Section \ref{LocalERR}, we give some properties on local error bounds under appropriate conditions. To find a solution of (\ref{defMS}), we employ a Levenberg-Marquardt algorithm in Section \ref{Algorithm}.  Some preliminary numerical results on synthetic data are reported in Section \ref{numTest}. Finally, we conclude the paper with drawing some remarks in Section \ref{Concl}.

\noindent{\bf Notation}. Let $\mathbb{R}^n$ be the space of $n$-dimensional real column vectors and $\mathbb{R}_+^n=\{x=(x_1,x_2,\cdots,x_n)^\top\in \mathbb{R}^n:x_i\geq 0,~~\forall~i=1,2,\cdots,n\}$. A vector of zeros in a real space of arbitrary dimension will be denoted by ${\bm 0}$. For any $x, y\in \mathbb{R}^n$, the Euclidean inner product is denoted by $x^\top y$, and the Euclidean norm $\|x\|$ is given by $\|x\|=\sqrt{x^\top x}$. For given $\mathcal{A}=(a_{i_1i_2\cdots i_m})\in \mathbb{T}_{m,n}$, if the entries $a_{i_1i_2\cdots i_m}$ are invariant under any permutation of their indices, then $\mathcal{A}$ is called a {\it symmetric} tensor. In particular, for every given index
$i\in [n]:=\{1,2,\cdots,n\}$, if an $(m-1)$-th order $n$-dimensional square tensor $\mathcal{A}_i:=(a_{ii_2\cdots i_m})_{1 \leq i_2,\cdots,i_m \leq n}$ is symmetric, then $\mathcal{A}$ is called a {\it semi-symmetric} tensor with respect to the indices $\{i_2,\cdots,i_m\}$. Denote the unit tensor in $\mathbb{T}_{m,n}$ by $\mathcal{I}=(\delta_{i_1\cdots i_m})$, where $\delta_{i_1\cdots i_m}$ is the Kronecker symbol
\begin{equation*}
\delta_{i_1\cdots i_m}=\left\{
\begin{array}{ll}
1,&\;\;{\rm if~}i_1=\cdots =i_m,\\
0,&\;\;{\rm otherwise}.
\end{array}
\right.\end{equation*}
With the notation \eqref{Axm-1}, we define ${\mathcal A}x^m = x^\top ({\mathcal A}x^{m-1})$ for ${\mathcal A}\in {\mathbb T}_{m,n}$ and $x\in {\mathbb R}^n$.
Moreover, ${\mathcal A}x^{m-2}$ denotes an $n\times n$ matrix whose $ij$-th component is given by
$$({\mathcal A}x^{m-2})_{ij}:= \sum^n_{i_3, ..., i_m = 1}a_{ij i_3 ... i_m}x_{i_3}\cdots x_{i_m} ,\quad i,\;j=1,2,\cdots,n.$$

\section{Preliminaries}\label{prelim}
In this section, we first summarize some definitions on tensors that will
be used in the coming analysis, and then introduce a new class of structured tensors.

\begin{definition}\label{def2.3}
Let $\mathcal{A} \in \mathbb{T}_{m,n}$. We say that $\mathcal{A}$ is

%\begin{flushleft}
\begin{itemize}
\itemindent 6pt
\item[(i)] a P-tensor (see \cite{SQ14}), if it holds that $\max\limits_{1\leq i\leq n}x_i(\mathcal{A}x^{m-1})_i > 0$ for any vector $x \in \mathbb{R}^n\backslash\{{\bm 0}\}$.
\item[(ii)] a strong P-tensor (see \cite{BHW16}), if it holds that $\max\limits_{1\leq i\leq n}(x_i-y_i)(\mathcal{A}x^{m-1}-\mathcal{A}y^{m-1})_i > 0$ for any vectors $x, y \in \mathbb{R}^n$ with $x\neq y$.
\item[(iii)] a positive definite tensor, if it holds that $\mathcal{A}x^{m}>0$ for any vector $x \in \mathbb{R}^n\backslash\{{\bm 0}\}$.
\end{itemize}
%\end{flushleft}

\end{definition}

\begin{definition}[\cite{WHQ18}]\label{ICPAii1}
Let ${\mathcal A}\in {\mathbb T}_{m,n}$. We say that ${\mathcal A}$ is a strictly positive definite tensor, if it holds that $
(x-y)^\top(\mathcal{A} x^{m-1}-\mathcal{A}y^{m-1})>0$ for any $x,y\in \mathbb{R}^n$ with $x\neq y$.
\end{definition}

From Definitions \ref{def2.3} and \ref{ICPAii1}, it is easy to see that a strictly positive definite tensor must be a strong P-tensor, and a strong P-tensor must be a P-tensor. However, a  strong P-tensor is not necessarily a strictly positive definite tensor, which will be shown in the following example.
\begin{exam}\label{exam2-1}
Let ${\mathcal A}=(a_{i_1i_2i_3i_4})\in\mathbb{T}_{4,2}$ with $a_{1111}=a_{2222}=a_{1122}=1$, $a_{1222}=-3$ and all other $a_{i_1i_2i_3i_4}=0$. For any $x,y\in \mathbb{R}^2$ with $x\neq y$, it is easy to see that
\begin{align}\label{Equat1}
(x_1-y_1)(\mathcal{A} x^3-\mathcal{A} y^3)_1=&\;(x_1-y_1)(x_1^3-y_1^3)-3(x_1-y_1)(x_2^3-y_2^3) \nonumber\\
&\;+(x_1-y_1)(x_1x_2^2-y_1y_2^2)
\end{align}
and \begin{equation}\label{Equat2}
(x_2-y_2)(\mathcal{A}  x^3-\mathcal{A}y^3)_2=(x_2-y_2)(x_2^3-y_2^3).
\end{equation}
We now consider two cases:

{\bf Case (i).} If $x_2\neq y_2$, it follows from \eqref{Equat2} that
\begin{align}\label{Equatt2}
(x_2-y_2)(\mathcal{A}  x^3-\mathcal{A}y^3)_2&=(x_2-y_2)^2(x_2^2+x_2y_2+y_2^2)\nonumber\\
&\geq \frac{1}{4}(x_2-y_2)^2(x_2^2-2x_2y_2+y_2^2) \nonumber \\
&= \frac{1}{4}(x_2-y_2)^4>0.
\end{align}

{\bf Case (ii).} If $x_2= y_2$, it immediately follows from the fact $x\neq y$ that $x_1\neq y_1$. In this situation, equation \eqref{Equat1} reduces to
\begin{align*}
(x_1-y_1)(\mathcal{A} x^3-\mathcal{A} y^3)_1&=(x_1-y_1)(x_1^3-y_1^3)+x_2^2(x_1-y_1)^2 \\
&\geq  \frac{1}{4}(x_1-y_1)^4+x_2^2(x_1-y_1)^2>0,
\end{align*}
where the first inequality can be derived by a similar technique used in \eqref{Equatt2}.
Hence, we know that $\max\limits_{1\leq i \leq 2}( x_i- y_i)(\mathcal{A} x^3-\mathcal{A} y^3)_i>0$ for any $x,y\in \mathbb{R}^2$ with $x\neq y$, which means that $\mathcal{A}$ is a strong P-tensor.

However, by taking $\bar x=(1,0)^\top$ and $\bar y=(0,-1)^\top$, we know that $(\bar x-\bar y)^\top (\mathcal{A}\bar x^3-\mathcal{A}\bar y^3)=-1<0$, which means that $\mathcal{A}$ is not a strictly positive definite tensor. Moreover, the tensor $\mathcal{A}$ in this example is not an $M$-tensor (see \cite{DW16}).
\end{exam}

\begin{definition}[\cite{FP03}]\label{SMonote}
A mapping $\Phi:\Omega\subseteq \mathbb{R}^n\rightarrow \mathbb{R}^n$ is said to be strictly monotone on $\Omega$ if and only if
$(x-y)^\top (\Phi(x)-\Phi(y)) > 0$ for all $x, y \in \Omega$ with $x\neq y$.
\end{definition}

For given $\mathcal{A}\in \mathbb{T}_{m,n}$ and $b\in \mathbb{R}^n$, defined by $\Phi(x)=\mathcal{A}x^{m-1}+b$. From Definitions  \ref{ICPAii1} and \ref {SMonote}, it is easy to see that $\Phi$ is strictly monotone on $\mathbb{R}^n$ if and only if the tensor $\mathcal{A}$ is strictly positive definite.

\begin{definition}[\cite{YLH17}]\label{NonsingularDef}
 Let $\mathcal{A}\in \mathbb{T}_{m,n}$. We say that $\mathcal{A}$ is singular, if $\mathcal{A}$ satisfies
 $$\{x\in \mathbb{R}^n\backslash\{{\bm 0}\}~|~\mathcal{A}x^{m-1}={\bm 0}\}\neq \emptyset.$$
 Otherwise, we say that $\mathcal{A}$ is nonsingular.
\end{definition}

Now, we introduce a new class of structured tensors, which includes the set of all $P$-tensors as its proper subset.
\begin{definition}\label{Efeef}
 Let ${\mathcal A}\in {\mathbb T}_{m,n}$. We say that ${\mathcal A}$ is a ${\rm Z^+}$-tensor, if there exists no $(x,t)\in(\mathbb{R}^n\backslash\{ {\bm 0}\})\times\mathbb{R}_+$ such that
\begin{equation}\label{equation1.2}
{\mathcal A}x^{m-1}+tx={\bm 0}.
\end{equation}
\end{definition}

It is obvious that, ${\mathcal A}$ is a ${\rm Z^+}$-tensor if and only if $\mathcal{A}$ has no non-positive Z-eigenvalue (see \cite{Qi05}). Furthermore, it can also be seen that if $\mathcal{A}$ is a ${\rm Z}^+$-tensor, then $\mathcal{A}$ is nonsingular.

\begin{proposition}\label{P-WR}
Let $\mathcal{A} \in \mathbb{T}_{m,n}$. If $\mathcal{A}$ is a P-tensor, then $\mathcal{A}$ is a ${\rm Z^+}$-tensor.
\end{proposition}

\begin{proof} Suppose that $\mathcal{A}$ is not a ${\rm  Z^+}$-tensor. Then, it follows from Definition \ref{Efeef} that there exists $(\bar x, \bar t)\in (\mathbb{R}^n\backslash\{ {\bm 0}\})\times\mathbb{R}_+$ such that (\ref{equation1.2}) holds. Therefore, we have
$$
\bar x_i(\mathcal{A}\bar x^{m-1})_i + \bar t\bar x^2_i=0, ~~~\forall ~i\in [n],
$$
which, together with $\bar t\geq 0$, implies that
\begin{equation}\label{TThr}
\max\limits_{1\leq i \leq n}\bar x_i(\mathcal{A}\bar x^{m-1})_i= -\min\limits_{1\leq i \leq n}\bar t \bar x_i^2\leq 0.
\end{equation}
Clearly, it contradicts to the given condition that $\mathcal{A}$ is a P-tensor. The proof is completed.
\qed\end{proof}

It was proved by Qi \cite{Qi05} that  Z-eigenvalues exist for an even order
real symmetric tensor $\mathcal{A}$, and $\mathcal{A}$ is {\it positive definite} (PD) if and only if all of its
Z-eigenvalues are positive, i.e., $\mathcal{A}$ is a ${\rm Z^+}$-tensor. Hence, in the symmetric tensor case, the concepts of PD, P and ${\rm  Z^+}$-tensors are identical. We also know that if $\mathcal{A}$ is an $m$-th order P-tensor, then $m$ must be even (see \cite{YY14}). So, there does not exist an odd order symmetric ${\rm Z^+}$-tensor. In the asymmetric tensor case, when $m=3$, we also claim that there does not exist a ${\rm Z^+}$-tensor. Indeed, for any $\mathcal{A}\in \mathbb{T}_{3,n}$, we know that, if $\lambda$ is a Z-eigenvalue of $\mathcal{A}$, then $-\lambda$ is also a Z-eigenvalue of $\mathcal{A}$, and $\mathcal{A}$ has odd Z-eigenvalues in the complex field (see \cite{Ful98}). Consequently, we know that $\mathcal{A}$  always has a real non-positive Z-eigenvalue, which means that $\mathcal{A}$ is not a ${\rm Z^+}$-tensor. So, even if in the asymmetric case, there does not exist a three order real ${\rm Z^+}$-tensor. However, in the case where $m\geq 4$, it is unclear that whether odd order real asymmetric ${\rm Z^+}$-tensors exist or not?

As proved in Proposition \ref{P-WR}, a P-tensor must be a ${\rm Z}^+$-tensor, but not
conversely. The following example is to show that a ${\rm Z^+}$-tensor is not necessarily a P-tensor.

\begin{exam}\label{exam11-1}
Let ${\mathcal A}=(a_{i_1i_2i_3i_4})\in\mathbb{T}_{4,2}$ with $a_{1111}=a_{1222}=a_{2111}=10$, $a_{1112}=a_{2122}=1$, $a_{1122}=2$, $a_{2112}=20$ and $a_{2222}=-8$. Then, for $x\in{\mathbb R}^2$, we have
 $$
 \mathcal{A}x^3=
 \left (
 \begin{array}{c}
 10x_1^3+x_1^2x_2
+2x_1x_2^2+10x_2^3 \\
10x_1^3+20x_1^2x_2
+x_1x_2^2-8x_2^3
\end{array}\right).
 $$
 We first claim that there is no $(x,t)\in(\mathbb{R}^2\backslash\{{\bm 0}\})\times\mathbb{R}_+$ such that \eqref{equation1.2} holds. Actually, if there exists a pair of $(\bar x,\bar t)\in(\mathbb{R}^2\backslash\{{\bm 0}\})\times\mathbb{R}_+$ such that \eqref{equation1.2} holds, then
 \begin{subnumcases}{\label{rree1}}
 10\bar x_1^3+\bar x_1^2\bar x_2
+2\bar x_1\bar x_2^2+10\bar x_2^3+\bar t\bar x_1=0,  \label{rree1a}\\
10\bar x_1^3+20\bar x_1^2\bar x_2
+\bar x_1\bar x_2^2-8\bar x_2^3+\bar t\bar x_2=0. \label{rree1b}
\end{subnumcases}
Without loss of generality, we suppose $\bar x_1\neq0$. It then follows from \eqref{rree1a} that
\begin{equation}\label{bart1}
\bar t=-(10\bar x_1^3+\bar x_1^2\bar x_2
+2\bar x_1\bar x_2^2+10\bar x_2^3)/\bar x_1.
\end{equation}
Consequently, substituting \eqref{bart1} into \eqref{rree1b} immediately yields
$$
10\bar x_2^4+10\bar x_1^3\bar x_2-10\bar x_1\bar x_1^3\bar x_2-10\bar x_1^4=0,
$$
which can be recast as
\begin{equation}\label{eqa}\bar s^4+\bar s^3-\bar s-1=(\bar s-1)(\bar s+1)(\bar s^2+\bar s+1)=0,
\end{equation}
where $\bar s=\bar x_2/\bar x_1\in \mathbb{R}$.  It is not difficult to observe that \eqref{eqa} has two different real roots. Correspondingly, if $\bar s=1$, then $\bar x_2=\bar x_1$ and \eqref{bart1} reduces to $\bar t=-23\bar x_1^2<0$, which contradicts  to $\bar t\in \mathbb{R}_+$. If $\bar s=-1$, then $\bar x_2=-\bar x_1$ and \eqref{bart1} can be specified as $\bar t=-\bar x_1^2<0$, which also contradicts  to $\bar t\in \mathbb{R}_+$. Hence, $\bar x_1=0$, which further implies that $\bar x_2=0$ by \eqref{rree1a}. Then, we have $\bar x=(\bar x_1, \bar x_2)^\top = {\bm 0}$, which contradicts to $\bar x\neq {\bm 0}$. Therefore, we know that $\mathcal{A}$ is a ${\rm Z^+}$-tensor. However, by taking $\tilde{x}=(1,-2)^\top\in \mathbb{R}^2\backslash\{{\bm 0}\}$, we have
$$
\max\left\{\tilde{x}_1(\mathcal{A}\tilde{x}^3)_1,\tilde{x}_2(\mathcal{A}\tilde{x}^3)_2\right\}=-64<0,
$$
which means that $\mathcal{A}$ is not a P-tensor.
\end{exam}

\begin{remark} Notice that ${\rm Z}^+$-tensor is a new concept introduced in this
paper. As showed in Proposition \ref{P-WR} and Example \ref{exam11-1}, ${\rm Z}^+$-tensor is a generalization of ${\rm P}$-tensor. Interestingly, the set of all ${\rm P}$-tensors includes many class of important structured tensors as its proper subset, for example, PD-tensors, even order strictly diagonally dominated tensors (\cite[Theorem 3.4]{YY14}), strong ${\rm P}$-tensors (\cite{BHW16}), even order Hilbert tensors (\cite[Theorem 1.1]{SQ14a}), even order strongly doubly
nonnegative tensors (\cite[Proposition 5.1]{LQ14}), even order strongly completely positive tensors (\cite[Definition 3.3]{DLQ15}), even order nonsingular ${\rm H}$-tensors with all positive diagonal entries
(\cite[Proposition 4.1]{DLQ15}), even order Cauchy tensors with mutually distinct entries of
generating vector (\cite[Corollary 4.4]{DLQ15}) and so on. If an even order ${\rm Z}$-tensor $\mathcal{A}$ is a ${\rm B}$-tensor \cite{SQ14}, then $\mathcal{A}$ is also a ${\rm P}$-tensor (\cite[Theorem 3.6]{YY14}). So, we believe that ${\rm Z}^+$-tensor is also a class of interesting structured tensors for future tensor analysis.
\end{remark}

\section{Existence of solutions to \eqref{defMS}}\label{ExistRes}
In this section, we focus on studying the existence of solutions for (\ref{defMS}) with the help of topological degree theory. We begin this section with presenting the following proposition of boundedness of the solution set of (\ref{defMS}).

\begin{proposition}\label{boundedness}
Let $\Lambda:=(\mathcal{A}_1,\mathcal{A}_2,\cdots,\mathcal{A}_{m-1})\in \mathbb{T}_{m,n}\times \mathbb{T}_{m-1,n}\times\cdots \times\mathbb{T}_{2,n}$. If the leading tensor $\mathcal{A}_1$ in $\Lambda$ is nonsingular, then the solution set ${\rm SOL}(\Lambda,b)$ of \eqref{defMS} is bounded for any $b\in \mathbb{R}^n$.
\end{proposition}

\begin{proof}
Suppose that ${\rm SOL}(\Lambda,b)$ is unbounded for some $\bar b\in \mathbb{R}^n$. Then there exists a sequence $\{x^r\}_{r=1}^\infty\subset{\rm SOL}(\Lambda,b)$ such that $\|x^r\|\rightarrow\infty$ as $r\rightarrow\infty$. Since $x^r\in {\rm SOL}(\Lambda,b)$, we have
\begin{equation}\label{Wett}\sum_{k=1}^{m-1}\frac{{\cal A}_k}{\|x^r\|^{k-1}}\left(\frac{x^r}{\|x^r\|}\right)^{m-k}=\frac{\bar b}{\|x^r\|^{m-1}}.
\end{equation}
Without loss of generality, we assume that $x^r/\|x^r\|\rightarrow \bar{x}$ as $r \rightarrow \infty$. It is clear that $\bar x\neq{\bm 0}$. Consequently, by letting $r\rightarrow\infty$ in (\ref{Wett}), we know $\mathcal{A}_1\bar x^{m-1}={\bm 0}$, which means that $\mathcal{A}_1$ is singular. It is a contradiction. Therefore, ${\rm SOL}(\Lambda,b)$ is bounded. The proof is completed.
\qed\end{proof}

When $m=2$, the problem \eqref{defMS} reduces to a linear system $Ax=b$, where $A\in \mathbb{R}^{n\times n}$ and $b\in \mathbb{R}^n$. It is well-known that $Ax=b$ is solvable for any $b\in \mathbb{R}^n$ if and only if $A$ is nonsingular, in this case, the solution of $Ax=b$ is unique. However, when $m\geq 3$, we could cannot ensure that $\mathcal{A}x^{m-1}=b$ is solvable for any $b\in \mathbb{R}^n$, even though $\mathcal{A}\in \mathbb{T}_{m,n}$ is nonsingular. For example, for the unit tensor $\mathcal{I}\in \mathbb{T}_{3,n}$, it is clear that $\mathcal{I}$ is nonsingular, but $\mathcal{I}x^2=-b$ has no real solution for any $b\in \mathbb{R}^n_+\backslash \{{\bm 0}\}$.

To study the existence of solutions for nonlinear equations, nonlinear complementarity
problems, and variational inequalities, a variety of concepts
of exceptional families of elements for continuous functions were
introduced in the literature (e.g., see, \cite{Isa01,IBK97,S84,ZH99,ZHQ99} and the references therein). Below, we introduce the definition of exceptional family of elements for a function.

\begin{definition}\label{def2.2}
Let $G:\mathbb{R}^n\rightarrow \mathbb{R}^n $ be a continuous function. We say that a set of elements $\{x^r\}_{r>0} \subset \mathbb{R}^n$ is an {\it exceptional family of elements} (in short, EFE)
for $G$, if the following conditions are satisfied:
\begin{itemize}
\item[(1)] $\|x^r\|\rightarrow \infty$ as $r \rightarrow \infty $,
\item[(2)] for each real number $r>0$, there exists a $\mu_r>0$ such that
$G(x^r)=-\mu_rx^r$.
\end{itemize}
\end{definition}

Let $\Omega$ be a bounded open set in $\mathbb{R}^n$ and $\partial \Omega$ represent the boundary of $\Omega$.  For a continuous function $U: \mathbb{R}^n \rightarrow \mathbb{R}^n$ and a vector $b\not\in U(\partial \Omega)$,  the degree of $U$ over $\Omega$ with respect to $b$ is defined, which is an integer and will be denoted by ${\rm deg} (U, \Omega, b)$. Here, we refer the reader to \cite{FFG95,LN78} for more details on degree theory. Now, we recall two fundamental theorems in the topological degree theory (e.g., see \cite[p. 23]{I06} and also \cite{LN78}), which play important roles in the proofs of our main results on the existenceness of solutions of \eqref{defMS}.

\begin{theorem}[Poincar\'{e}-Bohl Theorem]\label{th3}
Let $\Omega\subset \mathbb{R}^n$ be a bounded open set, $b\in \mathbb{R}^n$ and $U, V:\mathbb{R}^n\rightarrow \mathbb{R}^n$ be two continuous functions. If for all $x \in \partial\Omega$ the line segment $[U(x),V(x)]$ does not contain $b$,
then it holds that ${\rm deg}(U,\Omega,b) = {\rm deg}(V,\Omega,b)$.
\end{theorem}

\begin{theorem}[Kronecker's Theorem]\label{th4}
Let $\Omega\subset \mathbb{R}^n$ be a bounded open set, $b\in \mathbb{R}^n$ and $U:\mathbb{R}^n\rightarrow \mathbb{R}^n$ be a continuous function. If ${\rm deg}(U,\Omega,b)$ is defined and non-zero, then the equation $U(x)=b$ has a solution in $\Omega$.
\end{theorem}

By using the concept of EFE and Theorems \ref{th3} and \ref{th4}, we state and prove the following theorem.

\begin{theorem}\label{th2.6}
For a continuous function $F(x): \mathbb{R}^n\rightarrow\mathbb{R}^n$, there exists either a solution to $F(x)={\bm 0}$ or an exceptional family of elements for $F$.
\end{theorem}

\begin{proof}
For any real number $r>0$, we consider the spheres and open ball with radius $r$, respectively, i.e.,
$$S_r = \{x \in \mathbb{R}^n:\|x\| = r\} \quad {\rm and} \quad B_r = \{x \in \mathbb{R}^n:\|x\| < r\}.$$
Obviously, it can be seen that $\partial B_r = S_r$. We now consider the homotopy between the identity function $G(x)=x$ and $F$, which is  defined by
\begin{equation}\label{Hxt}
H(x,t) = tG(x) + (1-t)F(x), ~~\forall~~(x,t) \in S_r \times [0,1].
\end{equation}
Applying Theorems \ref{th3} and \ref{th4} to $H$ defined by \eqref{Hxt}, we have the following two situations:

(i)~There exist some $r>0$ such that $H(x,t)\neq {\bm 0}$ for any $x\in S_r$ and $t \in [0,1]$. In this situation, Theorem \ref{th3} implies that ${\rm deg}(F,B_r,{\bm 0})={\rm deg}(G,B_r,{\bm 0})$. Because ${\rm deg}(G,B_r,{\bm 0}) = 1$, we know ${\rm deg}(F,B_r,{\bm 0}) = 1$. Consequently, by Theorem \ref{th4}, we know that the ball $B_r$ contains at least one solution to the equation $F(x)={\bm 0}$.

(ii)~For each $r>0$, there exists a point $x^r \in S_r$ and a scalar $t_r \in [0,1]$ such that $H(x^r,t_r)={\bm 0}$. If $t_r = 0$, then $x^r$ is a solution of $F(x) = {\bm 0}$. Secondly, if $t_r = 1$, it is clear from the definition of $H(x,t)$ that
$t_r G(x^r) + (1-t_r)F(x^r) = x^r = {\bm 0}$, which contradicts the fact that $\|x^r\| = r > 0$. Finally, if $0 < t_r < 1$, it then follows from the definition of $H(x,t)$ that
$$  {\bm 0}=\frac{1}{1-t_r}[t_rG(x^r) + (1-t_r)F(x^r)]=\frac{t_r}{(1-t_r)}G(x^r) + F(x^r). $$
Letting $\mu_r = \frac{t_r}{1-t_r}$, we have $F(x^r) = -\mu_r G(x^r)= -\mu_r x^r$. Due to the fact $\|x^r\| = r$, it holds that $\|x^r\|\rightarrow \infty$ as $r \rightarrow \infty $. Thus, from Definition {\ref{def2.2}}, we know that $\{x^r\}$ is an exceptional family of elements
for $F$.
\qed\end{proof}

Throughout, for given $\Lambda:=(\mathcal{A}_1,\mathcal{A}_2,\cdots,\mathcal{A}_{m-1})\in \mathbb{T}_{m,n}\times \mathbb{T}_{m-1,n}\times\cdots \times\mathbb{T}_{2,n}$ and $b\in \mathbb{R}^n$, we denote by ${\rm SOL}(\Lambda,b)$ the solution set of \eqref{defMS}, i.e.,
\begin{equation}\label{solut}
{\rm SOL}(\Lambda,b):=\{x\in \mathbb{R}^n:F(x)={\bm 0}\},
\end{equation}
where the function $F : \mathbb{R}^n\rightarrow \mathbb{R}^n$ is given by
\begin{equation}\label{Feq}
F(x)={\cal A}_1x^{m-1}+{\cal A}_2x^{m-2}+\cdots+{\cal A}_{m-2}x^2+{\cal A}_{m-1}x - b.
\end{equation}

By Theorem \ref{th2.6}, we have the solutions existence theorem for \eqref{defMS} as follows.

\begin{theorem}\label{Exists}
Let $\Lambda:=(\mathcal{A}_1,\mathcal{A}_2,\cdots,\mathcal{A}_{m-1})\in \mathbb{T}_{m,n}\times \mathbb{T}_{m-1,n}\times\cdots \times\mathbb{T}_{2,n}$. Suppose that the leading tensor ${\mathcal A}_1$ in $\Lambda$ is a ${\rm Z^+}$-tensor. Then, the solution set ${\rm SOL}(\Lambda,b)$ of \eqref{defMS} is nonempty and compact  for any $b\in \mathbb{R}^n$.
\end{theorem}

\begin{proof}
We first prove the nonemptyness of ${\rm SOL}(\Lambda,b)$. Suppose, on the contrary, that ${\rm SOL}(\Lambda,b)=\emptyset$. Then, by Theorem \ref{th2.6}, we know that there exists an exceptional family of elements $\{x^r\}_{r > 0}$  for $F$ defined in (\ref{Feq}), i.e., $\{x^r\}_{r > 0}$ satisfies
$\|x^r\|\rightarrow \infty$ as $r \rightarrow \infty $, and for each real number $r>0$, there exists a scalar $\mu_r>0$ such that
$$
\sum_{k=1}^{m-1}{\cal A}_k(x^r)^{m-k}-b = -\mu_rx^r,
$$
which implies
\begin{equation}\label{TTr}
\sum_{k=1}^{m-1}\frac{{\cal A}_k}{\|x^r\|^{k-1}}\left(\frac{x^r}{\|x^r\|}\right)^{m-k}+ \frac{\mu_r}{\|x^r\|^{m-2}}\frac{x^r}{\|x^r\|} = \frac{b}{\|x^r\|^{m-1}}.
\end{equation}
It can be easily seen from \eqref{TTr} that $\{\mu_r/\|x^r\|^{m-2}\}$ is bounded. Without loss of generality, we assume that $x^r/\|x^r\|\rightarrow \bar{x}$ and $\mu_r/\|x^r\|^{m-2}\rightarrow \bar t$ as $r \rightarrow \infty$. It is clear that $\bar x\neq{\bm 0}$ and $\bar t\geq 0$. Since $\|x^r\|\rightarrow\infty$ as $r\rightarrow\infty$, by taking $r \rightarrow \infty$ in (\ref{TTr}), it holds that
$\mathcal{A}_1\bar{x}^{m-1} + \bar t\bar{x} = {\bm 0}$,
which contradicts to the given condition that $\mathcal{A}_1$ is a ${\rm Z^+}$-tensor. Therefore, ${\rm SOL}(\Lambda,b)$ is nonempty.

Now we prove the compactness of ${\rm SOL}(\Lambda,b)$. It is clear that ${\rm SOL}(\Lambda,b)$ is closed.
Moreover, $\mathcal{A}_1$  is nonsingular since it is a ${\rm Z}^+$-tensor. By Proposition \ref{boundedness}, we conclude that ${\rm SOL}(\Lambda,b)$ is bounded.

Hence, we obtain the desired result and complete the proof. \qed\end{proof}

As a byproduct of Theorem \ref{Exists}, we immediately obtain the existence of solutions of \eqref{mulineq}, which can be viewed as an interesting complement to the result discussed in \cite{DW16}.

\begin{corollary}\label{ExistGTeq}
Suppose that the coefficient tensor ${\mathcal A}$ in \eqref{mulineq} is a ${\rm Z^+}$-tensor. Then the system of tensor equations \eqref{mulineq} always has a solution for any $b\in \mathbb{R}^n$.
\end{corollary}

Moreover, we have the following uniqueness result on solution of \eqref{mulineq}.

\begin{theorem}\label{Unique}
Suppose that the coefficient tensor ${\mathcal A}$ in \eqref{mulineq} is a strong P-tensor, then the system of tensor equations \eqref{mulineq} always has a unique solution for any $b\in \mathbb{R}^n$.
\end{theorem}

\begin{proof}
It first follows from Definition \ref{def2.3} and Proposition \ref{P-WR} that ${\mathcal A}$ is a ${\rm Z^+}$-tensor. As a consequence of Corollary \ref{ExistGTeq}, we know that (\ref{mulineq}) has at least one solution  for any $b\in \mathbb{R}^n$. Suppose that both $\bar x$ and $\bar y$ are different solutions of \eqref{mulineq}, i.e., $\bar x\neq\bar y$ and $\mathcal{A}\bar x^{m-1}=\mathcal{A}\bar y^{m-1}=b$. It then can be easily seen that
$$\max\limits_{1\leq i \leq n}(\bar x_i-\bar y_i)(\mathcal{A} \bar x^{m-1}-\mathcal{A}\bar y^{m-1})_i=0,$$
which contradicts the condition that ${\mathcal A}$ is a strong P-tensor. Hence, we prove the desired result.
\qed\end{proof}

From Example \ref{exam2-1} and definitions of M-tensor (see \cite{Zhang3}) and strong P-tensor, we know that the set of all strong P-tensors and the set of all M-tensors do not contain each other. Hence, the existence result obtained above is different from Theorem 3.2 presented in \cite{DW16}. Additionally, the following example will show that the condition of ${\mathcal A}$ being a strong P-tensor in Theorem \ref{Unique} can neither be removed nor be replaced by the positive definiteness of tensor.

\begin{exam}\label{Exam3-1}
 Let $\mathcal{A}=(a_{i_1i_2i_3i_4})\in \mathbb{T}_{4,2}$ with $a_{1111}=2$, $a_{1122}=-3/2$, $a_{1222}=1$ and $a_{2222}=5/2$,  and all others $a_{i_1i_2i_3i_4}=0$. It is easy to see that for any $x\in \mathbb{R}^2\setminus\{{\bm 0}\}$, we have
 \begin{align*}
 \mathcal{A} x^4&=2x_1^4-(3/2)x_1^2x_2^2+x_1x_2^3+(5/2)x_2^4\\
 &\geq2x_1^4-(3/2)x_1^2x_2^2-(1/2)(x_1^2+x_2^2)x_2^2+(5/2)x_2^4\\
 &=(x_1^2-x_2^2)^2+x_1^4+x_2^4>0,
 \end{align*}
where the inequality comes from the fact that $2|x_1x_2|\leq x_1^2+x_2^2$ for any $x_1,x_2\in \mathbb{R}$, which means that $\mathcal{A}$ is a positive definite tensor. However, for $b=(6,20)^\top$, it is easy to see that $\mathcal{A}x^3=b$ can be specified as
\begin{numcases}{}
2x_1^3-(3/2)x_1x_2^2+x_2^3=6, \nonumber\\
 (5/2)x_2^3=20, \nonumber
\end{numcases}
 which implies $x_2=2$ and $2x_1^3-6x_1+2=0$. Consequently, it is not difficult to check that $2x_1^3-6x_1+2=0$ has two different positive real roots (denoted by $\hat{x}_1,\tilde{x}_1$) and a negative real root $\bar x_1$, so all $(\bar x_1,2)^\top$, $(\hat{x}_1,2)^\top$ and $(\tilde{x}_1,2)^\top$ are the solutions of $\mathcal{A}x^3=b$. On the other hand, if we take $b=(2,20)^\top$, then it is easy to see that the resulting tensor equations $\mathcal{A}x^{3}=b$ has only one solution $(\tilde{x}_1,2)^\top$ with $\tilde{x}_1<0$.
 \end{exam}

 As proved in \cite{DW16}, if the coefficient tensor $\mathcal{A}$ in (\ref{mulineq}) is a nonsingular M-tensor, then (\ref{mulineq}) has a unique positive solution for any given positive vector $b\in\mathbb{R}^n$. However, this example shows that, even if $\mathcal{A}$ in (\ref{mulineq}) is a positive definite tensor, we could not ensure that (\ref{mulineq}) has positive solutions for a positive vector $b\in \mathbb{R}^n$.

\section{Local error bound}\label{LocalERR}
In this section, we are going to study the local error bound condition that will play an important role in the convergence analysis of Levenberg-Marquardt algorithm presented in the next section. We begin this section by recalling the following definitions.

\begin{definition}[\cite{LP94,YF01}] \label{errorbound} Consider $W(x)={\bm 0}$, where $W:\mathbb{R}^n\rightarrow\mathbb{R}^n$ is a continuous function. Let $N$ be a subset of $\mathbb{R}^n$ such that $N\cap X\neq\emptyset$, where $X=\{x\in \mathbb{R}^n~|~W(x)={\bm 0}\}$. We say that $\|W(x)\|$ provides a local error bound on $N$ for $W(x)={\bm 0}$, if there exists a positive constant $c>0$ such
that
\begin{equation*}\label{errq}
\|W(x)\|\geq c ~{\rm dist}(x, X)
\end{equation*}
holds for any $x\in N$, where ${\rm dist}(x, X)=\inf_{y\in X}\|y-x\|$.
\end{definition}

It is well-known that, if the Jacobian matrix $W^\prime(\bar x)$ of $W$ at the solution $\bar x$ of $W(x)={\bm 0}$ is nonsingular, then $\bar x$ is an isolated solution. Hence, $\|W(x)\|$ provides a local error bound on some neighborhood of $\bar x$. However, the reverse claim is not necessarily true. The reader is referred to \cite{YF01} and the following example of tensor version.

\begin{exam}\label{exam3-1} Let ${\mathcal A}=(a_{i_1i_2i_3i_4})\in\mathbb{T}_{4,2}$ with $a_{1111}=a_{1222}=1$, $a_{1112}=a_{1122}=3$, $a_{2111}=a_{2222}=2$, $a_{2112}=a_{2122}=6$ and all other $a_{i_1i_2i_3i_4}=0$, and $b=(1,2)^\top$. Let $W(x)=\mathcal{A}x^{3}-b$. Then we have
$$
W(x)=
\left[
\begin{array}{c}
(x_1+x_2)^3-1\\
2(x_1+x_2)^3-2
\end{array}
\right].$$
For $x\in \mathbb{R}^2$, it holds that $\|W(x)\|=\sqrt{5}|(x_1+x_2)^3-1|$ and ${\rm dist}(x,X)=|x_1+x_2-1|/\sqrt{2}$, where $X=\{x\in \mathbb{R}^2:x_1+x_2=1\}$. Since $\lim_{\alpha\rightarrow 1}\frac{\alpha^3-1}{\alpha-1}=3$, it is easy to see that $|(x_1+x_2)^3-1|/|x_1+x_2-1|\geq 5/4$ for any $x\in N:=\{x\in \mathbb{R}^2~:~|x_1+x_2-1|\leq 1/2\}$. Consequently, we obtain
$$\|W(x)\|=\sqrt{5}~|(x_1+x_2)^3-1|\geq 5\sqrt{5}/4|x_1+x_2-1|=5\sqrt{10}/4~{\rm dist}(x,X)$$
for any $x\in N$. By taking $c=5\sqrt{10}/4$, we know that $\|W(x)\|$ provides a local error bound on $N$ for $W(x)={\bm 0}$.
However, it is clear that the Jacobian matrix $W^\prime(x)$ of $W$ is singular for any $x\in \mathbb{R}^2$.
\end{exam}

We now study the condition under which the local error bound holds. To this end, we make the following assumption on the function $F$ defined by (\ref{Feq}).

\begin{assum}\label{errorcondition}
For any given $\{x^{r}\}$ with $x^{r}\rightarrow \bar x\in X:={\rm SOL}(\Lambda,b)$, there exists a subsequence $\{x^{r_j}\}$ of $\{x^{r}\}$ and $i_0\in [n]$ such that
\begin{equation}\label{ErrTT}
\lim_{r_j\rightarrow\infty} \frac{F_{i_0}(x^{r_j})}{\|x^{r_j}-\bar x\|}\neq {\bm 0}.
\end{equation}
\end{assum}

\begin{theorem}\label{Errorres}
Let $\Lambda:=(\mathcal{A}_1,\mathcal{A}_2,\cdots,\mathcal{A}_{m-1})\in \mathbb{T}_{m,n}\times \mathbb{T}_{m-1,n}\times\cdots \times\mathbb{T}_{2,n}$. Suppose that the leading tensor $\mathcal{A}_1$ is nonsingular and Assumption \ref{errorcondition} holds. Then $\|F(x)\|$ provides a local error bound for \eqref{defMS}.
\end{theorem}

\begin{proof}
Suppose, on the contrary, that $\|F(x)\|$ is not a local error bound for \eqref{defMS}. Then there exists a positive number $\varepsilon$ and a sequence $\{x^{r}\}\subset \mathbb{R}^n$ such that $\|F(x^{r})\|\leq \varepsilon$ and
\begin{equation}\label{nonerror}
\frac{\|F(x^{r})\|}{{\rm dist}(x^{r},X)}\rightarrow 0.
\end{equation}
We first claim that $\{x^{r}\}$ is bounded. Otherwise, without loss of generalization, assume that $\|x^{r}\|\rightarrow \infty$ as $r\rightarrow\infty$. Since $\|F(x^{r})\|\leq \varepsilon$, we know that
\begin{align*}
\left\|\mathcal{A}_1(\bar x^{r})^{m-1}+\sum_{k=2}^{m-1}\frac{{\cal A}_k(\bar x^{r})^{m-k}}{\|x^{r}\|^{k-1}}-b/\|x^{r}\|^{m-1}\right\|=\frac{\|F(x^{r})\|}{\|x^{r}\|^{m-1}}\leq \frac{\varepsilon}{\|x^{r}\|^{m-1}},
\end{align*}
where $\bar x^{r}=x^{r}/\|x^{r}\|$. Without loss of generalization, we assume that $\bar x^{r} \rightarrow \bar x$ as $r\rightarrow\infty$. It is clear that $\|\bar x\|=1$. Letting $r\rightarrow\infty$ in the expression above leads to $\|\mathcal{A}_1\bar x^{m-1}\|=0$, which contradicts the condition that $\mathcal{A}_1$ is nonsingular. Since $\{x^{r}\}$ is bounded, we assume that $x^{r} \rightarrow x^*$ as $r\rightarrow\infty$. Furthermore, it is easy to see that $x^*\in X$. By Assumption \ref{errorcondition}, there exists a subsequence $\{x^{r_j}\}$ of $\{x^{r}\}$ and $i_0\in [n]$ such that (\ref{ErrTT}) holds. Consequently, there exists a subsequence $\{x^{r_{j_l}}\}$ of $\{x^{r_j}\}$ and $\varepsilon_0>0$, such that $|F_{i_0}(x^{r_{j_l}})|\geq \varepsilon_0\|x^{r_{j_l}}-x^*\|$, which implies $\|F(x^{r_{j_l}})\|\geq \varepsilon_0~{\rm dist}(x^{r_{j_l}},X)$ contradicting to (\ref{nonerror}).
\qed\end{proof}

To obtain a more checkable condition where the local error bound holds, we further make the following assumption.

\begin{assum}\label{Ass65}
Let $A\in \mathbb{T}_{2,n}$. For a given solution $\bar x$ of $F(x)={\bm 0}$, there exists a neighborhood $N(\bar x, {\bm r})$ of $\bar x$ and $c>0$ such that $\|A(x-\bar x)\|\geq c \|x-\bar x\|$ for any $x\in N(\bar x, {\bm r})$.
\end{assum}

It is obvious that if the matrix $A$ in Assumption \ref{Ass65} is nonsingular, then Assumption \ref{Ass65} holds.

\begin{theorem}\label{errorth2}
Let $\Lambda:=(\mathcal{A}_1,\mathcal{A}_2,\cdots,\mathcal{A}_{m-1})\in \mathbb{T}_{m,n}\times \mathbb{T}_{m-1,n}\times\cdots \times\mathbb{T}_{2,n}$. Suppose that $\mathcal{A}_{m-1}$ satisfies Assumption \ref{Ass65} on some neighborhood $N(\bar x,{\bm r})$ of $\bar x\in X:={\rm SOL}(\Lambda,b)$ and $G(x):=\sum_{k=1}^{m-2}{\cal A}_kx^{m-k}-b$ satisfies $\|G(x)-G(\bar x)\|=o(\|x-\bar x\|)$ as $x\rightarrow\bar x$. Then $\|F(x)\|$ provides a local error bound for \eqref{defMS} on $N(\bar x,{\bm r})$.
\end{theorem}

\begin{proof}
Since $\|F(x)\|=\|F(x)-F(\bar x)\|\geq \|\mathcal{A}_{m-1}(x-\bar x)\|-\|G(x)-G(\bar x)\|$, the desired result is immediately obtained from the given conditions.
\qed\end{proof}

\section{Levenberg-Marquardt algorithm for \eqref{defMS}}\label{Algorithm}
Notice that the model under consideration (see \eqref{Axm-1} and \eqref{defMS}) indeed is a structured system of nonlinear equations. In the literature, it has been well documented that Levenberg-Marquardt methods are efficient solvers for nonlinear equations. Therefore, in this section, we employ the most recent Levenberg-Marquardt algorithm proposed in \cite{ARC18} for \eqref{defMS} with a small modification on the LM parameter.

With the notation of $F(x)$ given by \eqref{Feq}, we can describe the Levenberg-Marquardt algorithm for \eqref{defMS} in Algorithm \ref{alg31}.

\begin{algorithm}[!htbp]
\caption{(Levenberg-Marquardt Algorithm for \eqref{defMS}).}\label{alg31}
\begin{algorithmic}[1]
\STATE Choose a positive integer $N_0$, $\mu_0>\bar \mu >0$, $\epsilon\in [1,2]$ and $0<p_0\leq p_1\leq p_2<1$. Let $x^{(0)}\in \mathbb{R}^n$
be starting points.
\WHILE{$\|(F^\prime(x^{(k)}))^\top F(x^{(k)})\|\neq 0$}
\STATE Let
\begin{equation}\label{lamk}\lambda_k=\frac{\mu_k \|F(x^{(k)})\|^\epsilon}{1+\|F(x^{(k)})\|}.\end{equation}
\STATE Compute $d_k$ by solving the following linear system of equations:
\begin{equation}\label{equation1}
[F^\prime(x^{(k)})^\top F^\prime(x^{(k)})+\lambda_k I]d=-F^\prime(x^{(k)})^\top F(x^{(k)}).
\end{equation}
\STATE Set
\begin{equation*}
\tau_k=\frac{\|F_{l(k)}\|^2-\|F(x^{(k)}+d_k)\|^2}{\|F(x^{(k)})\|^2-\|F(x^{(k)})+F^\prime(x^{(k)})d_k\|^2},
\end{equation*}
where $F_{l(k)}={\rm max}_{0\leq j\leq \chi_k}\{\|F(x^{k-j})\|\}$ and $\chi_k={\rm min}\{k,N_0\}$.
\STATE Update the next iterate $x^{(k+1)}$ by
 \begin{equation*}x^{(k+1)}:=\left\{
\begin{array}{cl}
x^{(k)}+d_k,&{\rm if}~\tau_k\geq p_0,\\
x^{(k)},&{\rm otherwise}.
\end{array}
\right.
\end{equation*}
\STATE Update the parameter $\mu_{k+1}$ by
\begin{equation*}\mu_{k+1}:=\left\{
\begin{array}{cl}
4\mu_k,&{\rm if}~\tau_k<p_1,\\
\mu_k,&{\rm if}~p_1\leq \tau_k\leq p_2,\\
{\rm max}\{\mu_k/4,\bar \mu\}, &{\rm otherwise}.
\end{array}
\right.
\end{equation*}
\ENDWHILE
\end{algorithmic}
\end{algorithm}

\begin{remark}
Note that the practical stopping criterion in Algorithm \ref{alg31} can be specified as $\|(F^\prime(x^{(k)}))^\top F(x^{(k)})\|\leq {\rm Tol}$, where `${\rm Tol}$' is a tolerance. Generally speaking, such a stopping criterion just leads to a stationary point of \eqref{defMS} with generic tensors, which may not always be a solution to \eqref{defMS}. Hence, in this paper, we use $\|F(x^{(k)})\|\leq {\rm Tol}$ in practice instead of the original one to ensure the obtained iterate $x^{(k)}$ being a solution to \eqref{defMS}, in addition to keeping the same stopping criterion when comparing Algorithm \ref{alg31} with the other methods. For the LM parameter $\lambda_k$ in \eqref{lamk}, we attach a parameter $\epsilon$ to $\|F(x^{(k)})\|$ for the purpose of improving the numerical performance of Algorithm \ref{alg31}.
\end{remark}

Below, we give the convergence results on Algorithm \ref{alg31}.  First, it is obvious that $F(x)$ defined by (\ref{Feq}) is continuously differentiable on $\mathbb{R}^n$, and Lipschitz continuous on any given bounded subset $\Omega$ in $\mathbb{R}^n$, i.e., there exists a constant $L>0$ such that
$$
\|F(x)-F(y)\|\leq L\|x-y\|,~~~\forall~x,y\in \Omega.
$$
Moreover, the Jocabian matrix function $F^\prime(x)$ of $F(x)$ is also Lipschitz continuous on any given bounded subset $\Omega$ in $\mathbb{R}^n$, i.e., there exists a constant $L>0$ such that
$$\|F^\prime(x)-F^\prime(y)\|\leq L\|x-y\|, ~~~~\forall~x,y\in \Omega.$$
Consequently, it is easy to see that
\begin{equation*}\label{LLTT} \|F(y)-F(x)-F^\prime(x)(y-x)\|\leq L\|y-x\|^2, ~~~~\forall~x,y\in \Omega.
\end{equation*}

When the leading tensor ${\mathcal A}_1$ in $\Lambda$ is nonsingular, the solution set ${\rm SOL}(\Lambda,b)$ is bounded for any $b\in{\mathbb R}^n$. Consequently, for any $x^{(0)}\in{\mathbb R}^n$, there is a scalar ${\bm c}_0>0$ such that the sequence $\{x^{(k)}\}$ generated by Algorithm \ref{alg31} belongs to the set ${\rm Lev}({\bm c}_0):=\{x\in{\mathbb R}^n | \|F(x)\|\leq {\bm c}_0\}$, which is bounded, if ${\mathcal A}_1$ in $\Lambda$ is nonsingular. As a consequence, both $F(x)$ and $F^\prime(x)$ are Lipschitz continuous on the set ${\rm Lev}({\bm c}_0)$, which implies the following global convergence theorem for Algorithm \ref{alg31}. The proof is skipped here for brevity and the reader is referred to \cite[Theorem 2.4]{ARC18} for similar details.

\begin{theorem}\label{glob con}
Suppose that the leading tensor ${\mathcal A}_1$ in $\Lambda$ is nonsingular. Then, Algorithm \ref{alg31} terminates in finite
iterations or satisfies
\begin{equation*}
\liminf_{k\rightarrow\infty}\|(F^\prime(x^{(k)}))^\top F(x^{(k)})\|=0.
\end{equation*}
\end{theorem}

Furthermore, when we suppose that the sequence $\{x^{(k)}\}$ generated by Algorithm \ref{alg31} is convergent to $x^*\in {\rm SOL}(\Lambda,b)$ and also lies in a neighborhood of $x^*$, Algorithm \ref{alg31} is then quadratically convergent (see \cite[Theorem 3.5]{ARC18}).

\begin{theorem}\label{local con}
Suppose that $\|F(x)\|$ provides a local error bound on $N(x^*,\varrho)$ for \eqref{defMS}, where $0<\varrho<1$ and $N(x^*,\varrho):=\{x\in \mathbb{R}^n~:~\|x-x^*\|\leq \varrho\}$. Then the sequence $\{x^{(k)}\}$ generated by Algorithm \ref{alg31} converges quadratically to a solution of \eqref{defMS}.
\end{theorem}
%\newpage

\section{Numerical experiments}\label{numTest}
As shown in Section \ref{Algorithm}, Algorithm \ref{alg31} (denoted by `LMA' throughout) has promising convergence properties. In this section, we further investigate its numerical behaviors on solving \eqref{defMS} with synthetic data. Due to the fact that for any $\mathcal{A}\in \mathbb{T}_{m,n}~(m\geq 3)$ there always exists a semi-symmetric tensor $\bar {\mathcal{A}}\in \mathbb{T}_{m,n}$ such that $\mathcal{A}x^{m-1}=\bar {\mathcal{A}}x^{m-1}$ for any $x\in \mathbb{R}^n$,  throughout this section, we consider the case where the coefficient tensors in \eqref{defMS} are semi-symmetric.

Note that the special case \eqref{mulineq} of \eqref{defMS} with ${\rm M}$-tensors and positive vector $b\in{\mathbb R}^n$ has been well studied in the recent literature. Here, we first consider model \eqref{mulineq} with different kinds of tensors (e.g., ${\rm M}$-tensors and general random tensors). Besides, we compare the proposed LMA with three benchmark algorithms, including the homotopy method \cite{Han17} (HM for short), the Newton-Gauss-Seidel method with one-step Gauss-Seidel iteration \cite{LXX17} (NGSM for short), and the quadratically convergent algorithm \cite{HLQZ18} (QCA for short). Then, we consider the generic model \eqref{defMS} and show the preliminary numerical results.

The code of the HM proposed by Han \cite{Han17} was downloaded from Han's homepage\footnote{http://homepages.umflint.edu/$\sim$lxhan/software.html}. The codes of the other three methods were written in {\sc Matlab} 2014a. Throughout, we employed the publicly shared tensor toolbox \cite{TensorT} to compute tensor-vector products and semi-symmetrization of tensors. All experiments were conducted on a DELL workstation computer with Intel(R) Xeon(R) CPU E5-2680 v3 @2.5GHz and 128G RAM running on Windows 7 Home Premium operating system.

\subsection{Solving tensor equations \eqref{mulineq}}\label{subs4-1}
As aforementioned, the system of tensor equations \eqref{mulineq} with M-tensors has been studied in recent works, e.g., see \cite{DW16,Han17,HLQZ18,LXX17,LLV18}. However, the coefficient tensor ${\mathcal A}$ may not be an M-tensor in some real-world problems (see \cite{Y18}). Therefore, we consider two cases where ${\mathcal A}$ is an M-tensor and a general tensor, respectively.

We first consider the system of tensor equations \eqref{mulineq} with semi-symmetric M-tensors. To generate an
M-tensor $\mathcal{A}$ in \eqref{mulineq}, we follow the way used in \cite{DW16}, that is, we randomly generate a nonnegative tensor ${\mathcal B}:=(b_{i_1i_2\cdots i_m})\in \mathbb{T}_{m,n}$, whose entries are uniformly distributed in $(0,1)$, and set ${\mathcal A}:=s{\mathcal I}-{\mathcal B}$ with
$$s=(1+\sigma)\cdot \max_{1\leq i\leq n}\left(\sum_{i_2,\cdots,i_m=1}^nb_{ii_2\cdots i_m}\right)\quad{\rm and}\quad \sigma=0.1.$$
Clearly, it follows from the fact
$$\rho(\mathcal B)\leq \max_{1\leq i\leq n}\left(\sum^n_{i_2,\cdots,i_m=1}b_{ii_2\cdots i_m}\right)$$
that $s>\rho({\mathcal B})$, which always ensures that ${\mathcal A}$ is a nonsingular M-tensor (see \cite{CQZ13,Qi05}). Then, we randomly generate the vector $b$ in \eqref{mulineq}, whose all entries are uniformly distributed in $(0,1)$. In this situation, we know that (\ref{mulineq}) has an unique positive solution \cite{DW16}.  For this case, we always take the starting point as $x^{(0)} = (1,1,...,1)^\top$ for all methods.

Note that we have shown that \eqref{mulineq} with a ${\rm Z}^+$-tensor has at least one solution for any $b\in{\mathbb R}^n$. From the definition of ${\rm Z}^+$-tensor, it is not easy to generate high-order and -dimensional ${\rm Z}^+$-tensors. Therefore, after the test of \eqref{mulineq} with M-tensors, we then consider a series of random semi-symmetric tensors, which are not necessarily ${\rm Z}^+$-tensors. Specifically, we generate random tensors ${\mathcal A}:=(a_{i_1i_2\cdots i_m})\in \mathbb{T}_{m,n}$, whose entries are uniformly distributed in $(-5,5)$. To ensure the problem under test has at least one solution, we first randomly generate a point $x^*\in\mathbb{R}^n$ whose entries are uniformly distributed in $(0,1)$, and then let $b = \mathcal{A}(x^*)^{m-1}$. It is clear that $x^*$ is a solution of \eqref{mulineq} with $(\mathcal{A},b)$. For this case, we choose the initial point $x^{(0)} =x^* +  (1,1,...,1)^\top$ for all methods.

Observing that $\lambda_k>0$ can lead to $F^\prime(x^{(k)})^\top F^\prime(x^{(k)}) + \lambda_kI$ in \eqref{equation1} being a positive definite matrix, we can gainfully compute the descent direction $d_k$ directly by the `{\sf left matrix divide}: $\backslash$', which is roughly same as the multiplication of the inverse of a matrix and a vector. For the linear subproblem of QCA in \cite{HLQZ18}, we employ the solver `{\sf bicg}' to it as suggested by the authors.

As suggested in \cite{Han17} and further verified in \cite{HLQZ18}, we implement all methods to solve the scaled system of \eqref{mulineq} instead of the original one, i.e., solving
\[\label{scale_sys}\hat{\mathcal A}x^{m-1}=\hat b\]
instead of directly finding solution to \eqref{mulineq}, where $\hat{\mathcal A}:={\mathcal A}/\omega$ and $\hat b:=b/\omega$ with $\omega$ being the largest value among the absolute values of components of ${\mathcal A}$ and the entries of $b$. When $\|\hat{\mathcal A}x^{m-1}-\hat b\|_2 \leq 10^{-12}$ is satisfied or the number of iterations exceeds $1000$, we terminate the methods and return solutions.

As suggested in \cite{ARC18}, we take $p_0=10^{-4}$, $p_1= 0.25$, $p_2 = 0.75$, $\bar \mu=10^{-8}$, $\mu_0=1$, $N_0=5$ for the proposed LMA. Additionally, note that $\epsilon \in [1,2]$ is a flexible parameter for LMA. It is unknown which value is better for the problem under consideration. Thus, we first investigate the numerical sensitivity of $\epsilon$ for our problem. Here, we test four values of $\epsilon$, i.e., $\epsilon=\{1,1,25,1,75,2\}$. Since all the data is generated randomly, we test $100$ groups of random data $({\mathcal A},b)$ for each pair of $(m,n)$ and report some numerical results. Throughout, `{\sf itr}' represents the average number of iterations; `{\sf time}' denotes the average computing time in seconds; `{\sf resi}' represents the average residual $\|\hat{\mathcal A}x^{m-1} - \hat{b}\|$ of the scaled system; and `{\sf sr}' denotes the success rate in $100$ groups of data sets $({\mathcal A},b)$ for each pair of $(m,n)$. Here, we think that the method is successful if the residual $\|\hat{\mathcal A}(x^{(k)})^{m-1} - \hat{b}\| \leq 10^{-12}$ in $1000$ iterations.

Tables \ref{tab1} and \ref{tab2} summarize the sensitivity of $\epsilon$ for M-tensors and general random-tensors, respectively. The results clearly show that LMA performs well for \eqref{mulineq} with M-tensors and positive vectors $b$. Especially, it seems from Table \ref{tab1} that $\epsilon=2$ is the best value in terms of taking the least average iterations, which suggests taking such a value for \eqref{mulineq} with M-tensors. Even though we can not verify that whether the coefficient tensors ${\mathcal A}$ are ${\rm Z}^+$-tensors, it can be seen from Table \ref{tab2} that the proposed LMA is also reliable with a high success rate for \eqref{mulineq} with general tensors. Moreover, it seems that $\epsilon=1$ is slightly better than the other three values for general cases.

\begin{table}[!htbp]
\caption{Numerical sensitivity of $\epsilon$ to LMA for \eqref{scale_sys} with M-tensors.}\label{tab1}
{%\footnotesize
\def\temptablewidth{1\textwidth}
\begin{tabular*}{\temptablewidth}{@{\extracolsep{\fill}}llll}\toprule
& $   \epsilon= 1$ && $  \epsilon = 1.25$\\
\cline{2-2} \cline{4-4}
$(m,n)$&itr / time / resi  / sr && itr / time / resi  / sr \\\midrule
$( 3,20)$ & 9.07 / 0.05  / 1.00$\times 10^{-13 }$/ 1.00 &&  9.00 / 0.05  / 8.46$\times 10^{-15 }$/ 1.00     \\
$( 3,50)$ & 11.00 / 0.06  / 1.15$\times 10^{-17 }$/ 1.00 &&  11.00 / 0.05  / 7.59$\times 10^{-18 }$/ 1.00 \\
$( 3,100)$ & 12.00 / 0.13  / 5.10$\times 10^{-17 }$/ 1.00 &&  12.00 / 0.12  / 5.50$\times 10^{-18 }$/ 1.00   \\
$( 4,50)$ & 14.00 / 0.38  / 2.38$\times 10^{-15 }$/ 1.00 &&  14.00 / 0.39  / 4.37$\times 10^{-17 }$/ 1.00  \\
$( 4,100)$ & 16.00 / 5.70  / 3.76$\times 10^{-17 }$/ 1.00 &&  16.00 / 5.77  / 1.16$\times 10^{-18 }$/ 1.00  \\
$( 5,20)$ & 14.01 / 0.21  / 2.32$\times 10^{-14 }$/ 1.00 &&  14.00 / 0.21  / 5.21$\times 10^{-16 }$/ 1.00   \\
$( 5,50)$ & 17.01 / 19.67  / 1.73$\times 10^{-13 }$/ 1.00 &&  17.00 / 19.74  / 1.10$\times 10^{-14 }$/ 1.00   \\\bottomrule
& $\epsilon = 1.75$ && $\epsilon = 2$\\
\cline{2-2} \cline{4-4}
$(m,n)$& itr / time / resi  / sr && itr / time / resi  / sr \\\midrule
$( 3,20)$ & 9.00 / 0.05  / 3.46$\times 10^{-17 }$/ 1.00 &&  8.94 / 0.04  / 2.76$\times 10^{-14 }$/ 1.00 \\
$( 3,50)$ & 10.20 / 0.05  / 2.26$\times 10^{-13 }$/ 1.00 &&  10.01 / 0.05  / 1.00$\times 10^{-13 }$/ 1.00 \\
$( 3,100)$ & 12.00 / 0.11  / 4.03$\times 10^{-18 }$/ 1.00 &&  11.98 / 0.10  / 1.71$\times 10^{-14 }$/ 1.00 \\
$( 4,50)$ & 13.52 / 0.37  / 2.78$\times 10^{-13 }$/ 1.00 &&  13.06 / 0.36  / 2.03$\times 10^{-13 }$/ 1.00 \\
$( 4,100)$ & 15.01 / 5.38  / 2.94$\times 10^{-13 }$/ 1.00 &&  15.00 / 5.36  / 9.46$\times 10^{-14 }$/ 1.00 \\
$( 5,20)$ & 13.37 / 0.20  / 1.78$\times 10^{-13 }$/ 1.00 &&  13.03 / 0.20  / 1.34$\times 10^{-13 }$/ 1.00 \\
$( 5,50)$ & 17.00 / 19.60  / 7.95$\times 10^{-17 }$/ 1.00 &&  16.91 / 19.50  / 6.40$\times 10^{-14}$/ 1.00 \\\bottomrule
\end{tabular*}}
\end{table}

\begin{table}[!htbp]
\caption{Numerical sensitivity of $\epsilon$ to LMA for \eqref{scale_sys} with general random tensors.}\label{tab2}
{%\footnotesize
\def\temptablewidth{1\textwidth}
\begin{tabular*}{\temptablewidth}{@{\extracolsep{\fill}}llllll}\toprule
& $   \epsilon= 1$ && $  \epsilon = 1.25$\\
\cline{2-2} \cline{4-4}
$(m,n)$&itr / time / resi  / sr && itr / time / resi  / sr \\\midrule
$( 3,20)$ & 10.02 / 0.05  / 8.08$\times 10^{-14}$  / 0.93 &&  10.14 / 0.05  / 3.14$\times 10^{-14}$ / 0.92  \\
$( 3,50)$ & 14.19 / 0.07  / 6.15$\times 10^{-14}$ / 0.88 &&  14.13 / 0.07  / 3.42$\times 10^{-14}$ / 0.85  \\
$( 3,100)$ & 16.57 / 2.78  / 3.97$\times 10^{-14}$ / 0.74 &&  17.93 / 2.91  / 8.06$\times 10^{-14}$ / 0.75 \\
$( 4,50)$ & 17.50 / 0.49  / 6.77$\times 10^{-14}$ / 0.84 &&  17.32 / 0.47  / 8.10$\times 10^{-14}$ / 0.78 \\
$( 4,100)$ & 20.71 / 149.76  / 4.30$\times 10^{-14}$ / 0.66 &&  22.42 / 158.99  / 9.51$\times 10^{-14}$ / 0.74 \\
$( 5,20)$ & 18.58 / 0.27  / 6.67$\times 10^{-14}$ / 0.77 &&  20.00 / 0.30  / 5.80$\times 10^{-14}$ / 0.72 \\
$( 5,50)$ & 33.75 / 38.23  / 5.33$\times 10^{-14}$ / 0.65 &&  33.42 / 37.78  / 6.46$\times 10^{-14}$ / 0.55\\\bottomrule
& $\epsilon = 1.75$ && $\epsilon = 2$\\
\cline{2-2} \cline{4-4}
$(m,n)$& itr / time / resi  / sr && itr / time / resi  / sr \\\midrule
$( 3,20)$ &  11.06 / 0.05  / 5.07$\times 10^{-14}$ / 0.94 &&  11.45 / 0.06  / 6.67$\times 10^{-14}$ / 0.95 \\
$( 3,50)$ &  17.09 / 0.08  / 6.90$\times 10^{-14}$ / 0.85 &&  18.12 / 0.08  / 5.56$\times 10^{-14}$ / 0.81 \\
$( 3,100)$ &  22.31 / 3.60  / 5.60$\times 10^{-14}$ / 0.75 &&  29.54 / 4.73  / 6.01$\times 10^{-14}$ / 0.70 \\
$( 4,50)$ & 26.41 / 0.73  / 6.11$\times 10^{-14}$ / 0.71 &&  29.02 / 0.78  / 7.10$\times 10^{-14}$ / 0.64 \\
$( 4,100)$ & 32.95 / 236.37  / 5.86$\times 10^{-14}$ / 0.63 &&  40.88 / 293.57  / 7.20$\times 10^{-14}$ / 0.57 \\
$( 5,20)$ & 20.55 / 0.30  / 5.90$\times 10^{-14}$ / 0.83 &&  17.13 / 0.26  / 3.62$\times 10^{-14}$ / 0.71 \\
$( 5,50)$ & 39.41 / 44.63  / 6.62$\times 10^{-14}$ / 0.63 &&  44.64 / 50.20  / 4.97$\times 10^{-14}$ / 0.58 \\\bottomrule
\end{tabular*}}
\end{table}

To compare the proposed LMA with the state-of-the-art solvers HM, QCA, and NGSM, we conduct seven pairs of $(m,n)$ and also randomly generate $100$ groups of $({\mathcal A},b)$ for every $(m,n)$. The average performance of the four methods is listed in Tables \ref{tab3} and \ref{tab4}. To show the evolutions of the residual $\|\hat{\mathcal A}x^{m-1} - \hat{b}\|$ with respect to iterations, we present two plots in Fig. \ref{Fig1}, which also shows the quadratic convergence rate of LMA. Moreover, we can see from Fig. \ref{Fig1} that both LMA and QCA are monotone algorithms, whereas both HM and NGSM are nonmonotone versions. Such a monotone behavior potentially supports that both LMA and QCA perform better than both HM and NGSM in many cases. Accordingly, we can conclude from Table \ref{tab3} and Fig. \ref{Fig1} that the proposed LMA is competitive to the promising HM \cite{Han17} and QCA \cite{HLQZ18} when dealing with \eqref{mulineq} with M-tensors and positive vectors $b$.

\begin{table}[!htbp]
\caption{Numerical comparison of the four algorithms for \eqref{scale_sys} with M-tensors.}\label{tab3}
{%\footnotesize
\def\temptablewidth{1\textwidth}
\begin{tabular*}{\temptablewidth}{@{\extracolsep{\fill}}llllll}\toprule
& LMA && HM \cite{Han17}\\
\cline{2-2} \cline{4-4}
$(m,n)$&itr / time / resi  / sr && itr / time / resi  / sr \\\midrule
$( 3,20)$ & 8.96 / 0.04  / 2.37$\times 10^{-14 }$/ 1.00 &&  8.92 / 0.08  / 1.29$\times 10^{-14 }$/ 1.00 \\
$( 3,50)$ & 10.00 / 0.05  / 8.00$\times 10^{-14 }$/ 1.00 &&  9.00 / 0.08  / 2.29$\times 10^{-14 }$/ 1.00 \\
$( 3,100)$ & 11.97 / 1.97  / 2.37$\times 10^{-14 }$/ 1.00 &&  9.00 / 2.74  / 3.16$\times 10^{-14 }$/ 1.00 \\
$( 4,50)$ & 13.05 / 0.36  / 2.18$\times 10^{-13 }$/ 1.00 &&  9.00 / 0.44  / 4.30$\times 10^{-13 }$/ 1.00 \\
$( 4,100)$ & 15.00 / 5.36  / 1.05$\times 10^{-13 }$/ 1.00 &&  11.00 / 139.33  / 2.63$\times 10^{-17 }$/ 1.00 \\
$( 5,20)$ & 13.09 / 0.19  / 9.58$\times 10^{-14 }$/ 1.00 &&  10.00 / 0.26  / 7.57$\times 10^{-14 }$/ 1.00 \\
$( 5,50)$ & 16.92 / 19.80  / 6.49$\times 10^{-14 }$/ 1.00 &&  11.00 / 22.00  / 1.22$\times 10^{-15 }$/ 1.00 \\
\bottomrule
& QCA \cite{HLQZ18} && NGSM \cite{LXX17}\\
\cline{2-2} \cline{4-4}
$(m,n)$& itr / time / resi  / sr && itr / time / resi  / sr \\\midrule
$( 3,20)$ &  8.07 / 0.04  / 4.03$\times 10^{-14 }$/ 1.00 &&  52.00 / 0.11  / 5.70$\times 10^{-13 }$/ 1.00 \\
$( 3,50)$ &  9.07 / 0.05  / 8.21$\times 10^{-14 }$/ 1.00 &&  56.41 / 0.13  / 5.66$\times 10^{-13 }$/ 1.00 \\
$( 3,100)$ &  9.82 / 2.16  / 5.78$\times 10^{-14 }$/ 1.00 &&  60.61 / 1.74  / 6.57$\times 10^{-13 }$/ 1.00 \\
$( 4,50)$ &  10.68 / 0.31  / 8.46$\times 10^{-14 }$/ 1.00 &&  73.28 / 0.97  / 4.19$\times 10^{-13 }$/ 1.00 \\
$( 4,100)$ &  11.58 / 79.73  / 6.86$\times 10^{-14 }$/ 1.00 &&  72.52 / 12.31  / 6.23$\times 10^{-13 }$/ 1.00 \\
$( 5,20)$ &  10.51 / 0.17  / 1.09$\times 10^{-13 }$/ 1.00 &&  80.78 / 0.59  / 4.34$\times 10^{-13 }$/ 1.00 \\
$( 5,50)$ &  12.14 / 13.71  / 1.28$\times 10^{-13 }$/ 1.00 &&  72.63 / 40.36  / 6.02$\times 10^{-13 }$/ 1.00 \\
\bottomrule
\end{tabular*}}
\end{table}

\begin{figure}[!htbp]
\includegraphics[width=0.499\textwidth]{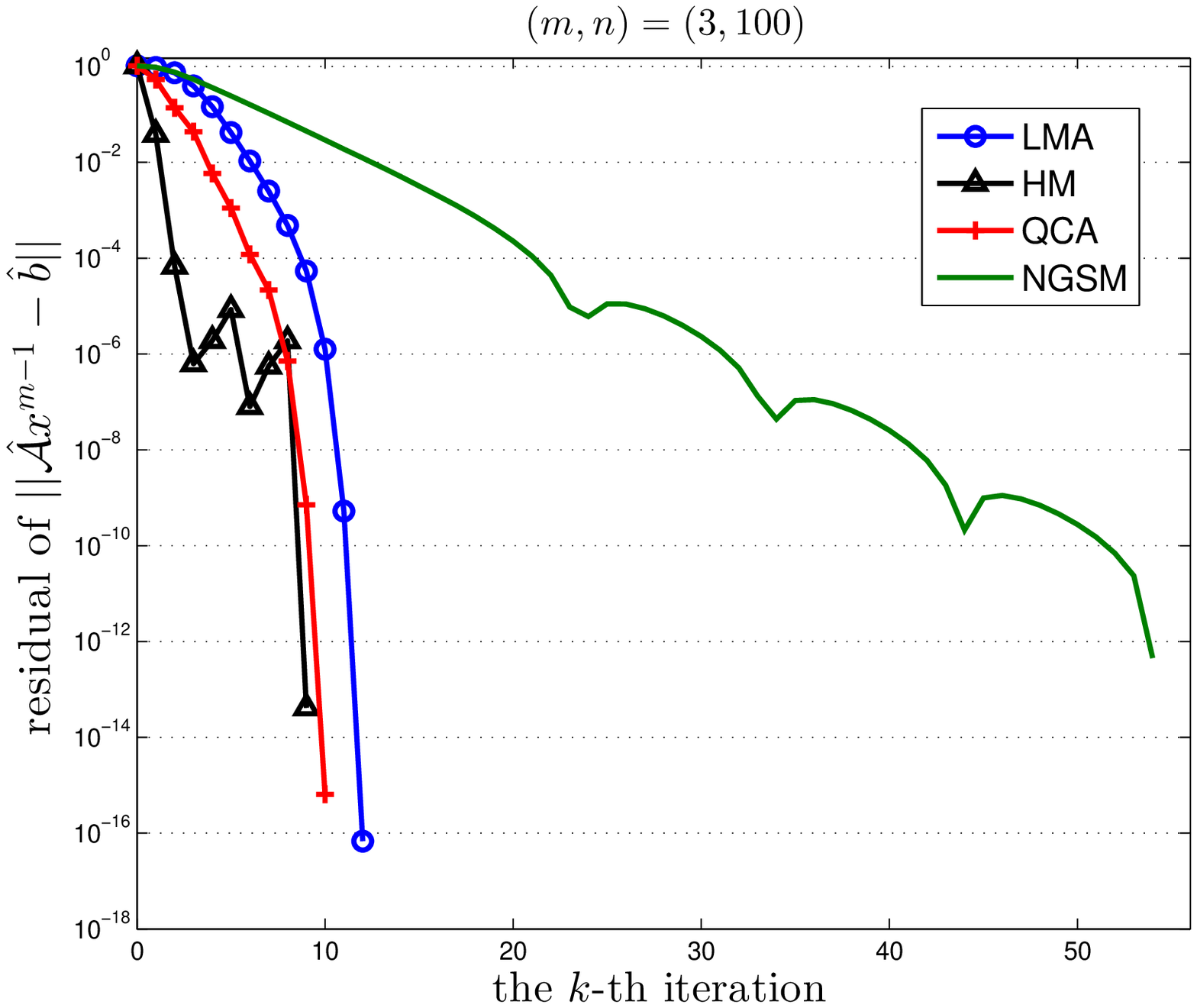}
\includegraphics[width=0.499\textwidth]{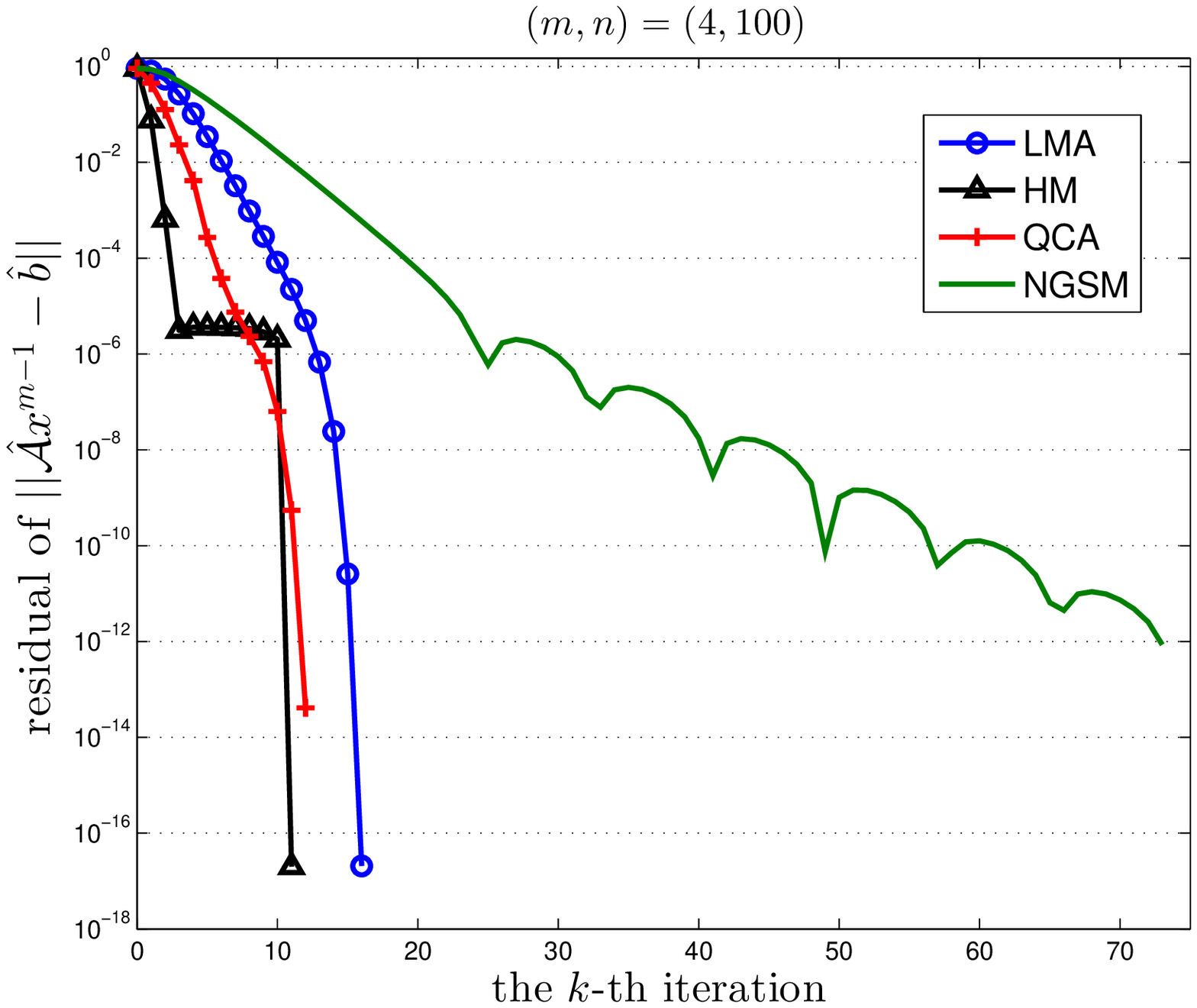}
\caption{Evolutions of the residual $\|\hat{\mathcal A}x^{m-1} - \hat{b}\|$ with respect to iterations.}
\label{Fig1}
\end{figure}

Notice that both HM \cite{Han17} and QCA \cite{HLQZ18} are tailored for \eqref{mulineq} with M-tensors and positive vectors $b$, and the convergence of NGSM \cite{LXX17} relies on the positive definiteness of $\mathcal{A}x^{m-2}$, which is a comparatively restrictive condition. Hence, it is not clear that whether HM, QCA, and NGSM are still able to find solutions of \eqref{mulineq} with general tensors. It is noteworthy that, in our experiments, we will use the symbol `--' to denote `{\sf itr}',  `{\sf time}',  `{\sf resi}',  `{\sf sr}' if the method can not get a solution satisfying $\|\hat {\mathcal A}(x^{(k)})^{m-1}-\hat b\|\leq 10^{-12}$ in $1000$ iterations. From the data reported in Table \ref{tab4}, we can see that HM, QCA, and NGSM fail to finding solutions of \eqref{mulineq} with generic random tensors under the preset tolerance. Actually, we have some more results on lower dimensional cases, e.g., $n=\{5,10,20\}$, which show that HM, QCA, and NGSM usually obtain satisfactory solutions in a very low probability (less than $10\%$). However, the proposed LMA can successfully find a solution with a high probability, which is a good news for finding solutions to generalized tensor equations.

\begin{table}
\caption{Numerical comparison of the four algorithms for \eqref{scale_sys} with general tensors.}\label{tab4}
\def\temptablewidth{1\textwidth}
\begin{tabular*}{\temptablewidth}{@{\extracolsep{\fill}}llll}\toprule
& LMA && HM / QCA / NGSM\\
\cline{2-2}\cline{4-4}
$(m,n)$& itr / time / resi / sr && itr / time / resi  / sr  \\\midrule
$(3,20)$ & 11.28 / 0.06  / 6.64$\times 10^{-14 }$/ 0.94 &&  -- / --  / -- / -- \\
$(3,50)$ & 14.43 / 0.07  / 1.11$\times 10^{-13 }$/ 0.75 &&  -- / --  / -- / --  \\
$(3,100)$ & 15.90 / 1.09  / 4.82$\times 10^{-14 }$/ 0.81 &&  -- / --  / -- / --  \\
$(4,50)$ & 16.45 / 0.53  / 6.63$\times 10^{-14 }$/ 0.74 &&  -- / --  / -- / --   \\
$(4,100)$ & 23.27 / 9.66  / 7.25$\times 10^{-14 }$/ 0.83 &&  -- / --  / -- / --  \\
$(5,20)$ & 18.11 / 0.29  / 7.36$\times 10^{-14 }$/ 0.83 &&  -- / --  / -- / --  \\
$(5,50)$ & 29.89 / 39.25  / 6.05$\times 10^{-14 }$/ 0.53 &&  -- / --  / -- / --  \\ \bottomrule
\end{tabular*}
\end{table}

\subsection{Solving generalized tensor equations \eqref{defMS}}\label{subs4-2}
In the last subsection, we can see that the proposed LMA is a probabilistic reliable solver for \eqref{mulineq} with M-tensors and general tensors. However, the theoretical and algorithmic results are mainly devoted to the generalized tensor equations \eqref{defMS}. Hence, we are further concerned with the numerical performance of LMA for \eqref{defMS}. Specifically, we consider the case where ${\mathcal A}_i\in{\mathbb T}_{(4-i+1),n},\;i=1,2,3$. As tested in Section \ref{subs4-1}, we consider two scenarios where ${\mathcal A}_i, i=1,2,3$ are M-tensors and generic random tensors, respectively. Moreover, we follow the way used in Section \ref{subs4-1} to generate ${\mathcal A}_i~ (i=1,2,3)$ and $b$. Throughout, the initial point $x^{(0)}$ is taken as $x^{(0)}=(1,1,\cdots,1)^\top$, and all parameters of LMA are taken as the values used in Section \ref{subs4-1}. The stopping criterion for LMA is set as
$$\| {\cal A}_1(x^{(k)})^{3}+{\cal A}_2(x^{(k)})^{2}+{\cal A}_{3}(x^{(k)}) - b\| \leq 10^{-6},$$
and the maximum iteration is taken as $1000$. Denote the order of ${\mathcal A}_i$ by $m_i$. In our experiments, we conduct five groups of $(m_1,m_2,m_3,n)$ with randomly generated $100$ groups of data sets $({\mathcal {A}_1, \mathcal {A}_2, \mathcal {A}_3}, b)$ for each scenario.

The results are listed in Table \ref{tab5}. For the case where $\mathcal{A}_i~(i=1,2,3)$ are M-tensors, it is easy to see from Table \ref{tab5} that LMA can always successfully find a solution of \eqref{defMS}. Even for the case where $\mathcal{A}_i~(i=1,2,3)$ are generic random tensors,  the proposed LMA can also find a solution to \eqref{defMS} in a relatively high probability.

\begin{table}[!htbp]
\caption{Numerical performances of LMA for \eqref{defMS} with M-tensors and general tensors.}\label{tab5}
{%\footnotesize
\def\temptablewidth{1\textwidth}
\begin{tabular*}{\temptablewidth}{@{\extracolsep{\fill}}llll}\toprule
& M-tensors && General tensors\\
\cline{2-2} \cline{4-4}
$(m_1,m_2,m_3,n)$&itr / time / resi / sr && itr / time / resi / sr \\\midrule
$(4,3,2,5)$&6.75 / 0.07 / 8.10$\times 10^{-8}$ / 1.00 && 9.39 / 0.09 / 8.13$\times 10^{-8 }$ / 0.90 \\
$(4,3,2,10)$&8.05 / 0.08 / 8.18$\times 10^{-8}$ / 1.00 && 16.08 / 0.14 / 9.55$\times 10^{-8}$ / 0.79 \\
$(4,3,2,20)$&9.90 / 0.09 / 6.13$\times 10^{-8}$ / 1.00 && 29.52 / 0.27 / 1.10$\times 10^{-7 }$ / 0.50 \\
$(4,3,2,50)$&12.00 / 0.40 / 2.48$\times 10^{-10}$ / 1.00 && 53.32 / 1.67 / 8.07$\times 10^{-8 }$ / 0.56 \\
$(4,3,2,100)$&13.61 / 4.94 / 2.84$\times 10^{-7}$ / 1.00 && 66.26 / 23.10 / 1.22$\times 10^{-7 }$ / 0.42 \\ \bottomrule
\end{tabular*}}
\end{table}

\section{Conclusion}\label{Concl}
In this paper, we considered a class of generalized tensor equations, which is an extension of the newly introduced tensor equations in \cite{DW16}. To study the existenceness of solutions, we first introduce a class of so-named ${\rm Z}^+$-tensors, which includes many well-known structured tensors such as P-tensors as its special type. With the help of degree theory, we showed that the solution set of GTEs is nonempty and compact when \eqref{defMS} has a leading ${\rm Z}^+$-tensor. Moreover, we established the local error bounds under some appropriate conditions and proposed a Levenberg-Marquardt algorithm to find a solution of \eqref{defMS} including its special case \eqref{mulineq}. Computational results show that the proposed LMA performs well for (generalized) tensor equations with M-tensors and generic random tensors. However, our algorithm still fails in some cases due to the starting point perhaps being far way the true solution of the problem. So, can we design structure-exploiting algorithms which are independent on the starting point? This is one of our future concerns.

\medskip
\begin{acknowledgements}
C. Ling and H. He were supported in part by National Natural Science Foundation of China (Nos. 11571087 and 11771113) and Natural Science Foundation of Zhejiang Province (LY17A010028).
\end{acknowledgements}

%\bibliographystyle{spmpsci}
%\bibliography{E:/Research/JabRef/HeArt,E:/Research/JabRef/Polynomial,E:/Research/JabRef/HeBook}

\end{document}